\documentclass[preprint, authoryear]{elsarticle}
\usepackage{amsmath,amsthm,amssymb}
\usepackage{graphicx}
\usepackage{url}
\usepackage[dvipsnames]{xcolor} % kivehető, plusz betűszíneket pl. OliveGreen tud

\newcommand{\revision}[1]{{\color{black}#1}}
\newcommand{\kron}{\otimes}

\newtheorem{theorem}{Theorem}
\newtheorem{lemma}[theorem]{Lemma}
\newtheorem{corollary}[theorem]{Corollary}
\newtheorem{definition}[theorem]{Definition}
\newtheorem{remark}[theorem]{Remark}
\newtheorem{example}[theorem]{Example}

\newcommand{\R}{\mathbb{R}}
\newcommand{\T}{\mathrm{T}}

\DeclareMathOperator*{\maximize}{maximize}
\DeclareMathOperator*{\st}{subject\;to}
\DeclareMathOperator{\inter}{int}
\DeclareMathOperator{\bd}{bd}

\usepackage{pgfplots}
\pgfplotsset{compat=1.15}
\usepackage{mathrsfs}
\usetikzlibrary{arrows}
\definecolor{ffffzz}{rgb}{1.,1.,0.6}
\definecolor{qqqqtt}{rgb}{0.,0.,0.2}

\begin{document}

\begin{center}
\large{\bf{The smallest mono-unstable convex polyhedron with point masses has 8 faces and 11 vertices}}  \\[5mm]

\normalsize
D{\'a}vid Papp \\[2mm]
{Department of Mathematics, North Carolina State University, Raleigh, NC, USA. \\ \texttt{https://orcid.org/0000-0003-4498-6417}} \\
Box 8205, NC State University, Raleigh, NC 27695-8205, USA  \\
\texttt{dpapp@ncsu.edu} \\[9mm]

Krisztina Reg{\H o}s \\[2mm]
{Department of Morphology and Geometric Modeling and MTA-BME Morphodynamics Research Group, Budapest University of Technology and Economics, Budapest, Hungary} \\
 M\H uegyetem rakpart 1-3., Budapest, Hungary, 1111  \\
\texttt{regos.kriszti@gmail.com} \\[9mm]

G{\'a}bor Domokos \\[2mm]
{Department of Morphology and Geometric Modeling and MTA-BME Morphodynamics Research Group, Budapest University of Technology and Economics, Budapest, Hungary} \\
 M\H uegyetem rakpart 1-3., Budapest, Hungary, 1111 \\
\texttt{domokos@iit.bme.hu}  \\[9mm]

S{\'a}ndor Boz{\'o}ki \\[2mm]
{corresponding author \\ 
Institute for Computer Science and Control (SZTAKI), Eötvös Loránd Research Network, Kende street 13-17, Budapest, Hungary, 1111; \\
Corvinus University of Budapest} \\ 
\texttt{bozoki.sandor@sztaki.hu} \\
\end{center}

\newpage

%\newpageafter{abstract}

\begin{abstract}
In the study of monostatic polyhedra, initiated by John H. Conway in 1966, the main question is to construct such an object with the minimal number of faces and vertices. By distinguishing between various material distributions and stability types, this expands into a small family of related questions. While many upper and lower bounds on the necessary numbers of faces and vertices have been established, none of these questions has been so far resolved.  Adapting an algorithm presented in \citep{Bozoki}, here we offer the first complete answer to a question from this family: by using the toolbox of semidefinite optimization to efficiently generate the hundreds of thousands of infeasibility certificates, we provide the first-ever proof for the existence of a monostatic polyhedron with point masses, having minimal number ($V=11$) of vertices (Theorem \ref{maintheorem}) and a minimal number ($F=8$) of faces. We also show that $V=11$ is the smallest number of vertices that a mono-unstable polyhedron can have in all dimensions greater than $1$ (Corollary \ref{thm:main-theorem-high-dim}).
\end{abstract}


\maketitle

\section{Introduction and the main result}

\subsection{Optimization and rigorous proofs}

%Nonlinear optimization, targeting global optima by numerical methods, has been a mainstay in classical applications of operations research,  (e.g., portfolio optimization, regression, packing problems).
%However, the scope of optimization has expanded in several ways:

Nonlinear optimization methodology has been impactful in both the core, classical application areas of operations research and in other interfacing disciplines due to flexibility with which it can be adapted to the needs of individual problems and areas:

\begin{itemize}

\item %Beyond classical problems, the
The toolkit of optimization has been successfully applied to problems that have not been previously regarded as optimization problems. Examples include establishing parameter identifiability and state observability in dynamical systems \citep{AugustPapachristodoulou2009}.

%HOVA? bounds for point configurations of minimum energy \citep{deLaat2020} and various packing problems \citep{Kurpel-2020,Lai2020}.

\item %Beyond seeking global optima,  advanced
Advanced heuristic optimization methods have been applied to improve locally optimal solutions for challenging optimization problems (e.g., \cite{LopezBeasley2011} and \cite{Lai2020}); simultaneously, optimization algorithms have been applied to compute tight bounds and confirming (numerically) that the best found local optimal solutions are approximately globally optimal (e.g., \cite{Kurpel-2020}). Polynomial optimization and semidefinite programming has achieved remarkable success in this area. Pertinent examples include the breakthrough in Kuperberg's (still open) problem on the number of infinite cylinders touching a ball \citep[Section 3.1.3]{Firsching2016}, in which a long-standing conjecture was refuted by finding an unexpected feasible solution and bounds for point configurations of minimum energy \citep{deLaat2020}.

\item %Beyond presenting numerical results, optimization
Optimization methods have been merged with the tools of computer assisted proofs to attain rigorous results.
The proof of rigorous bounds involves solving convex optimization problems that do not resemble the natural formulations used to obtain good feasible solutions \revision{and which facilitate the computation of rational solutions that can be verified in exact arithmetic} \citep{BomzeSchachingerUllrich2015,BomzeSchachingerUllrich2018}. The convexity of these auxiliary optimization problems means that the optimization methods yield \emph{easily and independently verifiable, rigorous proofs} of the bounds.
Such rigorous proofs are provided, for example, by \cite{BachocVallentin2008} 
for new bounds on the ``kissing problem’’ (the maximum number of non-intersecting unit spheres touching a fixed unit sphere in $n$ dimensions).
\end{itemize}

Our paper is in the same spirit: we apply optimization methodology to a problem (mechanical behavior of convex polyhedra) that has not been regarded as an optimization problem before, and we seek to improve the lower bound for the minimal number $V$ of vertices of a mono-unstable polyhedron. Ultimately, we succeed in finding the highest lower bound coinciding with the lowest upper bound, thus completely resolving the problem.  Our claim about the highest lower bound being $V=11$ is a sharp and rigorous result, all our claims can be verified by rational arithmetic.

\subsection{History of monostatic objects and the gap between upper and lower bounds}

Static balance points of a given object are points on its surface where, if supported on a horizontal plane, the object could be at rest. The numbers of various types of such balance points are often intuitively clear: for example a (fair) cubic dice has $S=6$
stable equilibrium positions on its faces \revision{and} $U=8$ unstable equilibrium positions at its vertices.
%and also has $H=12$ saddle-\revision{type} equilibria on its edges. 
Despite being associated with mechanical experiments,
the concept of static equilibrium may also be defined in purely geometric terms. Here we focus on equilibria associated with convex polyhedra and, following \citet{balancing}, we can write:
%%%%%%%%%%%%%%%%%%%%%%%%%%%%%%

\begin{definition}\label{def1}
Let $P \subseteq \R^{\revision{d}}$ be a $\revision{d}$-dimensional convex \revision{polytope}, let $\inter P$ and $\bd P$ denote its relative interior and boundary, respectively, and let $o \in \inter P$.  We say that $q \in \bd P$ is an \emph{equilibrium point} of $P$ with respect to $o$ if the \revision{hyper}plane \revision{$h$} through $q$ and perpendicular to the line segment $[o,q]$ supports $P$ at $q$. In this case $q$ is \emph{nondegenerate} if $\revision{h} \cap P$ is the (unique) \revision{$k$-dimensional face ($k=0,1, \dots d-1)$} of $P$ that contains $q$ in its relative interior. A nondegenerate equilibrium point $q$ is called \emph{stable} or \emph{unstable}, if $\dim (\revision{h} \cap P) = \revision{d-1}$, or $0$, respectively, \revision{otherwise we call it a \emph{saddle-type} equilibrium}. \revision{We denote the respective numbers of stable and unstable equilibria by $S$ and $U$}.
\end{definition}

%OLD  DEF1%%%%%%%%%%%%%%%%%%%%%%%
%\begin{definition}\label{def1}
%Let $P \subseteq \R^3$ be a $3$-dimensional convex polyhedron, let $\inter P$ and $\bd P$ denote its relative interior and boundary, respectively, and let $o \in \inter P$.  We say that $q \in \bd P$ is an \emph{equilibrium point} of $P$ with respect to $o$ if the plane \revision{$h$} through $q$ and perpendicular to the line segment $[o,q]$ supports $P$ at $q$. In this case $q$ is \emph{nondegenerate} if $\revision{h} \cap P$ is the (unique) vertex, edge, or face of $P$, respectively, that contains $q$ in its relative interior. A nondegenerate equilibrium point $q$ is called \emph{stable, saddle-type} or \emph{unstable}, if $\dim (\revision{h} \cap P) = 2$, $1$, or $0$, respectively and we denote their respective numbers by $S$, $H$, and $U$.
%\end{definition}
%\begin{remark}
%The definition for static equilibria of convex \revision{polyhedra in $\R^3$ and} polygons in $\R^2$ is analogous; however, in that case we only distinguish between (generic) stable equilibrium points in the relative interior of the edges and unstable equilibrium points at the vertices.
%\end{remark}
%%%%%%%%%%%%%%%%%%%%%%%%%%%
% Former Remark 3, RIP
%\begin{remark}
%The definition for static equilibria of smooth convex bodies is identical (via the support plane). In case of smooth bodies the three generic stability types are distinguished by the signs of the eigenvalues of the Hessian of the support function.
%\end{remark}
Throughout this paper we deal only with nondegenerate equilibrium points with respect to the center of mass $g$ of polyhedra, so, we have $o\equiv g$, in which case equilibrium points gain intuitive interpretation as locations on $\bd P$
where $P$ may be balanced if it is supported on a horizontal surface (identical to the \emph{support plane} mentioned in Definition \ref{def1}) without friction in the presence of uniform gravity.  We will describe cases
associated with uniform density (which we will refer to as \emph{homogeneous}) and cases where each vertex carries a unit mass (which we will refer to as \emph{0-skeletons}).

We call a convex body \emph{monostatic} if it has either one stable  or one unstable static equilibrium position. Convex bodies with $S=1$  stable position are also
referred to as \emph{mono-stable} and with $U=1$ unstable position as \emph{mono-unstable}, whereas convex bodies with $S=U=1$ (i.e. one stable and one unstable balance position) are called \emph{mono-monostatic}.  The geometry of such convex bodies appears to be enigmatic: the existence of a convex, homogeneous mono-monostatic convex body was conjectured by V.I. Arnold in 1995 \citep{lunch} and proved in 2006 \citep{VarkonyiDomokos}.

The rich variety of related discrete problems was opened by a brief note by \citet{Conway0}, who asked whether homogeneous, mono-stable \revision{polytopes} existed at all and conjectured that homogeneous tetrahedra cannot be mono-stable. (Throughout the paper, we use the shorthand \revision{\emph{polytope}} to mean a three-dimensional bounded convex polyhedron.) Both problems have been resolved in \citep{Conway}, where the authors presented a mono-stable, convex, homogeneous polytope with $F=19$ faces and $V=34$ vertices and proved that homogeneous polytopes with $V=F=4$ vertices and faces (i.e., tetrahedra) cannot be mono-stable. That is, for mono-stable, homogeneous polytopes we have $V,F > 4$. It immediately became intuitively clear that the essence of the problem is the 
rather substantial $\emph{gap}$ between the respective values of $F$ and $V$. 

Various related problems have been investigated since. In the case of homogeneous mono-unstable polytopes \citep{balancing}, the lower bound $V,F>4$ was established and an example of $V=F=18$ was provided. For convex mono-unstable $0$-skeletons, the lower bound $F \geq 6, V\geq 8$ has been established in \citep{Bozoki} and an example with $F = 8, V= 11$ was provided in \citep{h0gomboc}.
As we can see, in all investigated problems about monostatic polyhedra the gap between the lower bounds and the best known example exists, in fact, this gap appears to be a characteristic feature of this class of problems.

\begin{remark}
The size of the gap may differ for the number $F$ of faces and for the number $V$ of vertices. In \citep{balancing} a theory is presented how these gaps can be merged and quantified by a single scalar in a meaningful manner, however, this is beyond the scope of our current manuscript.
\end{remark}

\subsection{The main result: closing the gap for mono-unstable 0-skeletons}
Our goal in the paper is to close this gap in the case of mono-unstable $0$-skeletons by proving the following result:
\begin{theorem}\label{maintheorem}
The smallest vertex number for which there exists a mono-unstable 0-skeleton in 3 dimensions is $V=11$.
\end{theorem}
Since in \citep{h0gomboc} the authors presented examples with $V=11$ vertices, the essence of our paper is to prove the following:
\begin{theorem}\label{mainlemma}
No mono-unstable 0-skeletons exist with $V<11$ vertices.
\end{theorem}

This is an improvement of the lower bound shown in \citep{Bozoki}:
\begin{theorem}[Theorem 1 in \citep{Bozoki}] \label{bozoki}
For $V<8$, no mono-unstable 0-skeletons exist.
\end{theorem}

Since every 3-dimensional convex polytope with at least $11$ vertices has at least 8 faces, and the construction in \citep{h0gomboc} is a mono-unstable 0-skeleton with $8$ faces, this also proves that the minimum number of faces that a mono-unstable 0-skeleton may have is $8$.

Theorem \ref{mainlemma} is not just a quantitative generalization of Theorem \ref{bozoki}, and for two reasons:
first we note that (unlike Theorem \ref{bozoki}), due
to the existence of $V=11$ examples it can not be improved.
Second, the tools proving Theorem \ref{mainlemma} differ substantially from the tools used
in the proof of Theorem \ref{bozoki}: while the latter was proved using a randomized computer search for  certificates of infeasibility of certain polynomial systems, those tools have proved to be inefficient at going beyond the case $V=7$.
In the current paper, to resolve the cases $V=8,9,10$, we combine the techniques introduced in \citep{Bozoki} with semidefinite optimization to efficiently generate the hundreds of thousands of infeasibility certificates required to prove Theorem \ref{mainlemma}, demonstrating the superior power of these tools.

Beyond closing the gap for mono-unstable 0-skeletons in 3 dimensions, our computations also yielded an analogous result in dimensions two and higher:
\begin{corollary}\label{thm:main-theorem-high-dim}
Every mono-unstable 
0-skeleton in any dimension has at least 11 vertices.
\end{corollary}
We discuss this generalization in Section \ref{sec:dim-free}.

Our proof of Theorem \ref{mainlemma} is an easily verifiable computer-assisted proof generated using convex optimization. First, the statement of the theorem is translated to the unsolvability of several systems of polynomial inequalities following the work of \citet{Bozoki}; see Theorem \ref{thm:BDKR} below. Then the unsolvability of these systems is proved using a sufficient condition derived from linear algebra (Lemma \ref{lem:certificate}). The unsolvability certificates take the form of positive integer vectors that are generated using semidefinite optimization. The verification of these certificates can be carried out independently of the method they were generated with, simply by verifying that the generated integer vectors are indeed (strictly) feasible solutions of certain linear matrix inequalities.

Proving the infeasibility of systems of polynomial equations and inequalities and the equivalent problem of rigorously certifying lower bounds of polynomials on semialgebraic sets (that is, solution sets of polynomial inequalities) are becoming a fundamental tool in automated system verification and theorem proving \citep{deKlerkMaharryPasechnikRichterSalazar2006,deKlerk2016,MagronConstantinidesDonaldson2017,UhlmannWang2021}, with applications in various areas of engineering, operations research, and statistics, including
power systems engineering (optimal power flow) \citep{JoszMaeghtPanciaticiGilbert2015,GhaddarMarecekMevissen2016}, signal processing \citep{Dumitrescu2017}, and design of experiments \citep{Papp2012}. It has also been a particularly popular and successful technique in computer-assisted \emph{geometric} theorem proving. Although computer-assisted proofs in geometry go back at least to the celebrated work of \citet{Hales2005}, more recent work combining polynomial optimization and convex optimization techniques have resulted in \emph{easily verifiable} computer-assisted proofs of, for example, lower or upper bounds on optimal packings and other point configurations; see, e.g., \citep{BachocVallentin2008,BallingerBlekhermanCohnGiansiracusaKellySchurmann2009,Firsching2016,DostertDeLaatMoustrou2021} to name only a few.

Most of these works rely on semidefinite optimization to compute certifiable global lower bounds of polynomials (or trigonometric polynomials) over semialgebraic sets in a manner similar to our approach, and can also be interpreted as applications of Lasserre's moment relaxation of polynomial optimization problems \citep{CamposMisenerParpasEJOR2019,deKlerk2010,Lasserre2001,Laurent2009}. One major difference in our approach is that instead of formulating the problem as a single large-scale polynomial optimization problem, we work with a large number of small instances of polynomial optimization problems involving only quadratic polynomials whose infeasibility can be proved at the lowest level of the Lasserre hierarchy.

\revision{The question of existence of solutions of systems of quadratic inequalities is also directly related to the celebrated S-lemma, which in its original form characterizes consistent systems of \emph{two} not necessarily convex quadratics. See \citep{PolikTerlaky2007} for precise statements and an approachable and extensive review on this subject. Direct generalizations (without additional assumptions) are known to be impossible (as shown in the article cited above), although there is some literature on similar statements for larger systems of quadratics, e.g., \citep{JeyakumarLiWoolnough2021}, usually under assumptions that make the original proofs generalize to larger systems. To the best of our understanding, these results are not applicable to the systems that arise in our study.}

In what follows, we shall present the details of our proof without further references to the theory of moment relaxations, \revision{algebraic geometry}, or polynomial optimization, and derive it instead from basic linear algebraic principles.

\section{Proof of the main result}
In this section, we prove Theorem \ref{mainlemma} (and by extension, Corollary \ref{thm:main-theorem-high-dim}) by certifying the infeasibility of a number of systems of polynomial equations and inequalities--an idea introduced in \citep{Bozoki}. We rely on the same necessary condition of the existence of mono-unstable $0$-skeletons as in that paper, but improve on the search for infeasibility certificates using semidefinite optimization. \revision{Throughout, we shall assume (without loss of generality) that the center of mass $g$ is at the origin of our coordinate system.} The essential results we need from \citep{Bozoki} are summarized below in Theorem \ref{thm:BDKR}.
\begin{theorem}\label{thm:BDKR}
Let $r_i \in \R^d, (i=1,\dots,V)$ be the vertices of a convex polytope with
\begin{subequations}\label{eq:shadow}
\begin{equation}
\sum_{i=1}^V r_i = 0. \label{eq:shadow-eq}
\end{equation}
Then $r_1$ is the only unstable vertex of the $0$-skeleton of this polytope if and only if for every $i\in\{2,\dots,V\}$ there exists a $j_i\in\{1,\dots,i-1\}$ satisfying
\begin{equation}
(r_i - r_{j_i})^\T r_i \leq 0 \qquad i=2,\dots,V. \label{eq:shadow-ineq}
\end{equation}
\end{subequations}
\end{theorem}

The geometric intuition behind this theorem is as follows. With the polytope's center of mass $g$ at the origin $0$ by Eq.~\eqref{eq:shadow-eq}, the interpretation of the inequality \eqref{eq:shadow-ineq} is that the line segment connecting vertices $r_i$ and $r_{j_i}$ forms a right or obtuse angle with the line segment that connects vertex $r_i$ and the center of mass. Therefore, if we attempt to balance the polytope on vertex $r_i$ by placing it on a horizontal support plane with the center of mass vertically above vertex $i$, then the $[r_i, r_{j_i}]$ line segment will be either below the support plane (if strict inequality holds in \eqref{eq:shadow-ineq}) or incident to it (in the case of equality). In either case, the horizontal support plane does not intersect the polytope in $r_i$ alone, and therefore the polytope is not at a (nondegenerate) unstable equilibrium. That $r_1$ is an unstable vertex in this case follows from the aforementioned fact that every 0-skeleton has at least one unstable vertex.

 Thus, to establish that no mono-unstable 0-skeletons with $V$ vertices exist, it is sufficient to prove that for all $(V-1)!$ choices of $j_i\in\{1,\dots,i-1\}$ \mbox{$(i=2,\dots,V)$}, the system of inequalities and equations \eqref{eq:shadow} has no non-zero solutions.

\subsection{Tractable infeasibility certificates}

Whether a system of polynomial inequalities is solvable over the reals is algorithmically decidable in the real number model using (for example) quantifier elimination methods \citep{Tarski1951,Renegar1992}. However, with their (at least) exponential running time in the number of variables, these exact procedures are prohibitively expensive to apply to our problem. Additionally, they do not produce easily checkable \emph{infeasibility certificates}. This means that if they conclude that the polynomial system in question does not have a solution, it is difficult to independently and efficiently verify that this conclusion was correct, leaving doubts about the validity of the computer-assisted proof.
In a similar fashion, we cannot rely on numerical QCQP solvers or other global optimization software to ``verify'' that the systems \eqref{eq:shadow} have no solutions. Even ignoring possible errors arising from the use numerical methods instead of exact arithmetic and the exponential running time (in the number of variables), these solvers also do not produce the infeasibility certificates we need for our rigorous proof.

Our approach to verify the unsolvability of all $(V-1)!$ systems \eqref{eq:shadow} is to look for efficiently computable and efficiently verifiable infeasbility certificates based on sufficient (but not necessary) conditions of infeasibility. The system \eqref{eq:shadow} can be simplified by expressing, say, each coordinate $r_{V,k}\,(k=1,\dots,d)$ of the last vertex $r_V$ as a linear combination of the other variables using \eqref{eq:shadow-eq} and substituting them back to \eqref{eq:shadow-ineq}, to obtain an equivalent system of $V-1$ homogeneous quadratic inequalities in $n := d(V-1)$ variables with integer coefficients. For such systems of inequalities, we can use the following sufficient condition of infeasibility:
\begin{lemma}\label{lem:certificate}
Consider the system of homogeneous quadratic inequalities
\begin{equation}\label{eq:homquad}
r^\T Q_i r \leq 0 \quad i=1,\dots,m,
\end{equation}
wherein each $Q_i \in \R^{n\times n}$ is a real symmetric matrix. If there exist nonnegative rational numbers $c_1,\dots,c_m$ such that the matrix $\sum_{i=1}^m c_i Q_i$ is positive definite, then \eqref{eq:homquad} does not have any non-zero solutions.
\end{lemma}
\begin{proof}
Leaving out the requirement that $c$ be rational, the statement is an immediate consequence of the definition of positive definiteness. Regarding rationality, if there exists a (not necessarily rational) nonnegative real vector $c$ such that $\sum_{i=1}^m c_i Q_i$ is positive definite, then its positive components can be perturbed to (arbitrarily close) positive rational numbers, resulting in a nonnegative rational vector satisfying the same.
\end{proof}

To expound on the application of Lemma \ref{lem:certificate} to Theorem \ref{thm:BDKR}, we first explicitly write the system \eqref{eq:shadow-ineq} in the form \eqref{eq:homquad}, ignoring the equations \eqref{eq:shadow-eq}. Stacking the coordinate vectors $r_1,\dots,r_V$ of the vertices into a single column vector $r\in\R^{Vd}$, each matrix $Q_i$ can be described as a $V \times V$ block matrix made up of blocks of size $d\times d$. Collecting the coefficients of the homogeneous quadratic
\begin{equation*} (r_i - r_{j_i})^\T r_i = \sum_{k=1}^{d} r_{i,k}^2 - r_{i,k}r_{j_i,k},
\end{equation*}
and keeping in mind that (by definition) each $Q_i$ is a symmetric matrix, we see that that for each $i = 2,\dots V$, the $(i,i)$-th block is the identity matrix $I_{d\times d}$, while the $(i,j_i)$-th and $(j_i,i)$-th blocks are $-\frac{1}{2}I_{d\times d}$. All other blocks are zero. In summary, the inequalities \eqref{eq:shadow-ineq} can be written as $r^\T Q_i r \leq 0  \,\,\, (i=2,\dots,V)$ with
\begin{equation}\label{eq:Q-concrete}
Q_i = (E^V_{ii} - E^V_{ij_i}/2 - E^V_{j_ii}/2) \kron I_d,
\end{equation}
\revision{where $E^V_{ij}$ is the $V\times V$ unit matrix whose $(i,j)$-th entry is $1$ and all other entries $0$}, and $\kron$ denotes the Kronecker product.

To complete the formulation, we backsubstitute $r_V = -\sum_{i=1}^{V-1} r_i$ from \eqref{eq:shadow-eq} into our system. Since $j_i < i \leq V$ for each $i$, this only affects $Q_2, \dots, Q_{V-1}$ by eliminating the $V$-th block row and column (which are all zeros). We can also determine the new $Q_V$ in closed form: since
\[ \revision{r_V^\T r_V - r_V^\T r_{j_V} = \Big(\sum_{i,j=1}^{V-1} r_i^\T r_j\Big) + \Big(\sum_{i=1}^{V-1} r_i^\T r_{j_V}\Big)},\]
we have
%\[ Q_V = S \kron I_{d}, \;\text{ where } \;  S = (s_{ij})_{i,j=1,\dots,V-1}\, \text{ with } \, s_{ij} = \begin{cases} 2 & i = j = j_V \\ 3/2 & \\ 1 & \text{otherwise} \end{cases}. \]
\[ Q_V = \Big(J_{V-1} + \frac{1}{2}\sum_{i=1}^{V-1} \left( E^{V-1}_{i, j_V} + E^{V-1}_{j_V, i} \right) \Big) \kron I_{d}, \]
where $J_{V-1}$ is the $(V-1)\times(V-1)$ all-ones matrix and \revision{$E$ denotes unit matrices as defined above}.

% r_V^2 - r_V r_{j_V}
% = (\sum r_i)^2 + (\sum r_i) r_{j_V}
% = (\sum_{i,j=1}^{V-1} r_i r_j) + (\sum_{i=1}^{V-1} r_i r_{j_V})

We can find coefficients $c_i$ satisfying the condition in Lemma \ref{lem:certificate} using semidefinite optimization. In the following, we use the common shorthand $A\succcurlyeq B$ for the relation that the matrix $A-B$ is positive semidefinite.

\begin{corollary}\label{cor:SDP}
Let $Q_1,\dots,Q_m \in \R^{n\times n}$ real symmetric matrices, and consider the following semidefinite optimization problem:
\begin{equation}\label{eq:SDP}
\begin{aligned}
&\maximize_{z\in\R,c\in\R^m}\quad    && z\\
&\st &&\sum_{i=1}^m c_i Q_i \succcurlyeq z I\\
&&& \|c\|_2 \leq 1\\
&&& c_i \geq z \qquad i=1,\dots,m.
\end{aligned}
\end{equation}
The optimal value of \eqref{eq:SDP} is positive if and only if there exist positive rational numbers $c_1,\dots,c_m$ such that the matrix $\sum_{i=1}^m c_i Q_i$ is positive definite. Any rational feasible solution $(z,c)$ of \eqref{eq:SDP} with $z>0$ is a certificate for the non-existence of non-zero solutions of the system \eqref{eq:homquad}.
\end{corollary}

%It is immediate that \eqref{eq:SDP} has an optimal solution $(z^*, c^*)$ satisfying $z^*>0$ if and only if the condition in Lemma \ref{lem:certificate} can be satisfied. The optimality of the solution is not critical, nor is the norm constraint: any feasible solution to the first and last constraint is sufficient, the goal of the optimization is to find a solution in the interior of the set of solutions. In this manner, if $z^*$ is sufficiently positive, then the $c^*$ component of a numerical solution (even if only approximately optimal and approximately feasible) can be ``rounded'' to nearby rational $c$ that satisfies Lemma \ref{lem:certificate}.

Semidefinite optimization models such as \eqref{eq:SDP} are typically solved using numerical methods, which compute solutions that may be only approximately feasible or approximately optimal. This is of no concern for our proof, as we only need to find a componentwise positive feasible solution to \eqref{eq:SDP}. (The purpose of the norm constraint on $c$ is to ensure that the problem is bounded, and can safely be violated.) As long as the maximum value of $z$ is sufficiently positive (compared to the precision of the floating point computation), the approximately feasible and approximately optimal solution returned by a numerical semidefinite optimization method already serves as a rigorous proof of the non-existence of solutions of \eqref{eq:homquad} by \mbox{Lemma \ref{lem:certificate}}.

Thus, to prove that every 3-dimensional mono-unstable 0-skeleton has at least 11 vertices, we run the following algorithm: for every choice of $(j_2,\dots,j_{10}) \in \{1\} \times \{1,2\} \times \cdots \times \{1,\dots,9\}$, we transform the corresponding system \eqref{eq:shadow} to an equivalent system of homogeneous quadratic inequalities of the form \eqref{eq:homquad} by expressing each $r_{10,k}\; (k=1,2,3)$ as a linear combination of the other variables using \eqref{eq:shadow-eq} and substituting them back to \eqref{eq:shadow-ineq}, and then we solve the corresponding semidefinite optimization problem \eqref{eq:SDP} using a numerical method to prove that the system \eqref{eq:homquad} has no non-zero solutions.

The independently verifiable computer-generated proof is the list of positive rational vectors $c$ (one for each permutation) returned by the semidefinite optimization algorithm. (The $z$ component of the optimal solution is irrelevant as long as it is positive, and is not part of the infeasibility certificate.) The correctness of these vectors can be verified efficiently in rational arithmetic: it suffices to verify that the matrix $\sum_{i=1}^m c_i Q_i$ is positive definite, which can be carried out in polynomial time in rational arithmetic, say, using the $LDL^\T$ form of Cholesky decomposition or by verifying the positivity of the determinant of each leading principal submatrix.

%\newpage
\begin{example}
Let $V=10$ and $j_i = i-1$ for $i=2,3,\ldots,10$. Then $\sum_{i=2}^{10} c_i Q_i =$

\begin{eqnarray*}
\hspace*{-11mm}
\left(
%\HUGE(
\normalsize{
\begin{matrix}
c_{10} & -\frac{1}{2} c_{2}+c_{10} & c_{10} & c_{10} & c_{10}   \\[2mm]  
-\frac{1}{2} c_{2}+c_{10} & c_{2}+c_{10} & -\frac{1}{2} c_{3}+c_{10} & c_{10} & c_{10}   \\[2mm]  
c_{10} & -\frac{1}{2} c_{3}+c_{10} & c_{3}+c_{10} & -\frac{1}{2} c_{4}+c_{10} & c_{10}   \\[2mm]  
c_{10} & c_{10} & -\frac{1}{2} c_{4}+c_{10} & c_{4}+c_{10} & -\frac{1}{2} c_{5}+c_{10}   \\[2mm] 
c_{10} & c_{10} & c_{10} & -\frac{1}{2} c_{5}+c_{10} & c_{5}+c_{10}   \\[2mm]  
c_{10} & c_{10} & c_{10} & c_{10} & -\frac{1}{2} c_{6}+c_{10}   \\[2mm] 
c_{10} & c_{10} & c_{10} & c_{10} & c_{10}   \\[2mm]  
c_{10} & c_{10} & c_{10} & c_{10} & c_{10}   \\[2mm]  
\frac{3}{2} c_{10} & \frac{3}{2} c_{10} & \frac{3}{2} c_{10} & \frac{3}{2} c_{10} & \frac{3}{2} c_{10}    \\  
\end{matrix}} 
\right.
\end{eqnarray*}
\begin{eqnarray*}
\hspace*{30mm}
\normalsize
\left.
\begin{matrix}
 c_{10} & c_{10} & c_{10} & \frac{3}{2} c_{10}   \\[2mm]  
 c_{10} & c_{10} & c_{10} & \frac{3}{2} c_{10}   \\[2mm]  
 c_{10} & c_{10} & c_{10} & \frac{3}{2} c_{10}   \\[2mm]  
 c_{10} & c_{10} & c_{10} & \frac{3}{2} c_{10}   \\[2mm] 
 -\frac{1}{2} c_6+c_{10} & c_{10} & c_{10} & \frac{3}{2} c_{10}   \\[2mm]  
 c_{6}+c_{10} & -\frac{1}{2} c_{7}+c_{10} & c_{10} & \frac{3}{2} c_{10}   \\[2mm] 
 -\frac{1}{2} c_{7}+c_{10} & c_{7}+c_{10} & -\frac{1}{2} c_{8}+c_{10} & \frac{3}{2} c_{10}   \\[2mm]  
 c_{10} & -\frac{1}{2} c_{8}+c_{10} & c_{8}+c_{10} & -\frac{1}{2} c_{9}+ \frac{3}{2} c_{10}   \\[2mm]  
 \frac{3}{2} c_{10} & \frac{3}{2} c_{10} & -\frac{1}{2} c_{9}+\frac{3}{2} c_{10} & c_{9}+2 c_{10}   \\  
\end{matrix}
%\Bigg)
\right) \kron I_{3}
\end{eqnarray*}

The vector of coefficients $(c_2, c_3, \ldots, c_{10}) = (1, 4, 7, 8, 8, 7, 5, 4, 2)$, see also in  the last row of the supplemented csv file, makes the $9\times 9$ matrix above positive definite.
\end{example}

\subsection{A dimension-free view}\label{sec:dim-free}

Another look at the explicit form of the $Q_i$ matrices from \eqref{eq:Q-concrete} reveals a surprising fact. Since for every real symmetric matrix, the eigenvalues of $A \kron I_d$ are the same as the eigenvalues of $A$, only the multiplicities of the eigenvalues differ, the matrix
\[\sum_{i=1}^m c_i Q_i = \sum_{i=1}^m c_i ((E^V_{ii} - E^V_{ij_i}/2 - E^V_{j_ii}/2) \kron I_d) = \left(\sum_{i=1}^m c_i (E^V_{ii} - E^V_{ij_i}/2 - E^V_{j_ii}/2)\right) \kron I_d\] is positive definite if and only if $\sum_{i=1}^m c_i (E^V_{ii} - E^V_{ij_i}/2 - E^V_{j_ii}/2)$ is positive definite. That is to say, our approach of using Lemma \ref{lem:certificate} to prove the infeasibility of the sytem \eqref{eq:shadow} can only work if the system \eqref{eq:shadow} has no solution \emph{for any dimension $d$}. Although we are mainly concerned with 3-dimensional polytopes, this means that we have also shown that no mono-unstable polytopes with fewer than 11 vertices exist \emph{in any embedding dimension} (stated in the Introduction as Corollary \ref{thm:main-theorem-high-dim}).

It is important to note \revision{(also from the S-lemma's point of view, referred in the end of Section 1)} that the converse is not true: our proof technique will fail (a simple infeasibility certificate $c$ in Corollary \ref{cor:SDP} will not be found) if a mono-unstable polytope with $V$ vertices exists in \emph{any} dimension $d$, but this failure does not immediately reveal the dimensions $d$ for which a polytope or any other solution to the system \eqref{eq:shadow} exists. 
In particular, it is not difficult to show that for $d=1$ the system \eqref{eq:shadow} does not have a solution for any $V$ and any choice of $j_2,\dots,j_V$. Yet, our technique can only prove this for $V<11$, since for dimensions $d\geq 2$ a solution exists.

On the same note, following the idea of \citet{Dawson1985}, a simple perturbation argument makes it clear that if for some choice of $d$, $V$, and $j_i$ there exists a strictly feasible solution to \eqref{eq:shadow} that corresponds to the vertices of a convex polyhedron, then the same is true for the same choice of $V$ and $j_i$ in all higher dimensions. Since a mono-monostatic convex polygon with $11$ vertices was recently constructed by \citet[Figure~1]{h0gomboc}, this proves the existence of a mono-unstable $0$-skeleton with $11$ vertices in all dimensions $d \in \{2,\dots,11\}$.

Applying our proof technique in the one-dimensional case is also equivalent to what is sometimes referred to as the \emph{Gram matrix method} in convex algebraic geometry. Notice that the inequalities \eqref{eq:shadow-ineq} only depend on the vertex coordinates $r_i$ through their inner products $r_i^\T r_j$; furthermore, the center-of-mass equation \eqref{eq:shadow-eq} can also be equivalently written in terms of these inner products as
\[ 0 = \Big\|\sum_{i=1}^V r_i\Big\|_2^2 = \sum_{i=1}^V\sum_{j=1}^V r_i^\T r_j. \]
Therefore, if we consider the \emph{Gram matrix}
\[ R = (r_i^\T r_j)_{i,j=1,\dots,V}\]
associated with the vectors $r_1,\dots,r_V$, then the non-existence of a mono-unstable polytope in $d$ dimensions is implied, by virtue of Theorem \ref{thm:BDKR}, by the non-existence of a $V \times V$ symmetric, positive semidefinite, rank-$d$ matrix $R$ with the following two properties:
\begin{enumerate}
	\item $\sum_{i,j=1}^V R_{ij} = 0$.
	\item For each $i\in\{2,\dots,V\}$ there exists a $j \in \{1,\dots,i-1\}$ for which $R_{ii} \leq R_{ij}$.
\end{enumerate}
%(Once again, the converse of the statement does not hold: the existence of a rank-$d$ Gram matrix does not imply that we can find corresponding $r_1,\dots,r_V$ that are the vertices of a convex polyhedron.)
The rank condition in the above statement is computationally challenging, but to prove the non-existence of solutions, it is sufficient to prove that no positive semidefinite matrix (of any rank) satisfying the above two conditions exists. This leads to another semidefinite programming formulation, which is equivalent to \eqref{eq:SDP}.

\subsection{Implementation}

The algorithm was implemented using the semidefinite programming solver CSDP \citep{csdp}, interfaced using Mathematica, on a standard desktop computer. The enumeration and solution of the $9!$ optimization problems took approximately half an hour.

\revision{The numerical solutions (specifically, the near-optimal, near-feasible vectors $c$ obtained from CSDP) are rational numbers that were confirmed using rational arithmetic to be feasible solutions of \eqref{eq:SDP}. Since the numerical solutions are rational numbers represented in double precision floating-point arithmetic, this is an easy and efficient step, which does not involve any rational numbers with large bit sizes. Thus, in principle, these floating-point vectors themselves could be used as the rational certificates in Corollary \ref{cor:SDP}. Purely for the ease of dissemination and verification, these vectors were then further scaled up to positive integer vectors (recall that a positive multiple of an infeasibility certificate is also an infeasibility certificate), and then ``rounded'' to integer vectors with smaller components, once again confirming in rational arithmetic that the resulting vectors are still correct infeasibility certificates for their respective systems.}

The list of the computed \revision{integer} $c$ vectors certifying the unsolvability of the systems \eqref{eq:shadow} can be found in the public repository \url{https://github.com/dpapp-github/mono-unstable}. This, along with the proof of Theorem \ref{thm:BDKR}, serves as the independently verifiable computer-assisted proof of Theorem \ref{mainlemma}.

\section{Discussion}

\subsection{Improving the results about the mechanics of polyhedra}
Our result fixes the minimally necessary number of vertices as $V=11$ for a mono-unstable 0-skeleton and, via the theorem of Steinitz \citep{Steinitz2}, also the minimal number of faces. By the construction in \citep{h0gomboc} we know that these bounds are sharp, i.e. that the $V=11,F=8$ values  are not only necessary but also sufficient to create a mono-unstable 0-skeleton.

In the original problem we did not specify the number $S$ of stable equilibria, i.e. the question was to find the minimal number of vertices (and faces) for $U=1$, for \emph{any} value of $S$.  Since the constructions in \citep{h0gomboc} have $S=2$ or $S=3$ stable equilibria, consequently, for any $S>3$, the question remains open.

While we expect that for very modest increase of $S$ (e.g. for $S=4$) the same combinatorial values $(F,V)=(8,11)$ may remain valid as necessary and sufficient, this will definitely change as $S$ is further increased. In fact, the theorem of Steinitz also states that for $V=11$ vertices the maximal number of faces is $F=18$. So, if we prescribe $S=19$ stable equilibria (beyond the single unstable one), we will certainly have to have $V>11$ vertices.

While this problem is slightly different in nature from the one resolved in the current manuscript, our method could still be used to explore it: although our approach was primarily designed for improving lower bounds, it also aids the search for monostatic polyhedra. For example, in the problem studied in this paper, the method certified not only the infeasibility of all $9!$ systems for $V=10$, but also the infeasibility of the majority of the systems for $V=11$. This makes it easier to conduct a targeted search for mono-unstable $0$-skeletons with $V\geq 11$ vertices. In case of $V=11$, 
for many of the systems not certified by the method, it was straightforward to find a solution that also corresponded to the vertices of a convex polytope (for $d=3$) or polygon (for $d=2$). We expect that, to some extent, this could also be done for higher values of $V$.

\subsection{Improving the algorithm}
From the point of view of semidefinite optimization, our algorithm could be certainly made more efficient. For example, the number of cases to individually certify could be substantially lowered by eliminating those which are equivalent under a change of variables. The exploitation of such symmetries may dramatically lower the number of certificates to compute and may be an indispensable ingredient in resolving other problems in this area, where the number of cases is too large to allow their complete enumeration.
Since inequality (\ref{eq:shadow-ineq}) implies $\left|{r}_i \right| \leq \left|{r}_{j_i} \right|$ \citep[Lemma 2]{Bozoki}, moreover, $\left|{r}_i \right| < \left|{r}_{j_i} \right|$ for different nonzero vectors, it also induces a transitive binary relation on the vertices. This decreases the number of relevant cases from $(V-1)!$ 
% 362880 for V=10
to the number of rooted trees on $V$ vertices.
% 719 for V = 10
However, this advantage is coupled  with the drawback that the correctness and completeness of the computer-generated certificates becomes much harder to verify. Since our aim is that our results remain verifiable as simply as possible, we keep all the cases in the supplementary files. \revision{Although we hope that in this manner, the interested reader will find the verification of the certificates to be a very simple matter using any computer algebra system, we have also supplied an independently written computer code (purposely written in a different programming language than the code that generates the certificates).}

\section*{Acknowledgements}
The authors are grateful to the editor and the anonymous reviewers for their 
constructive comments. Special thanks to the anonymous reviewer for suggesting to explore the connection to the Gram matrix method.
The authors thank the anonymous reviewers and editors of our previous paper \citep{Bozoki} for their valuable suggestions. 

\paragraph{Funding} DP: This material is based upon work supported by the National Science Foundation under Grant No. DMS-1847865. GD, KR:
The support of the NKFIH Hungarian Research Fund grant 134199 and of the NKFIH Fund TKP2021 BME-NVA, carried out at the Budapest University of Technology and Economics, is kindly acknowledged.
KR: This research has been supported by the program ÚNKP-22-3 by ITM and NKFIH. The gift representing the Albrecht Science Fellowship is gratefully appreciated. SB: The research has been supported in part by the TKP2021-NKTA-01 NRDIO grant.

\bibliographystyle{elsarticle-harv}
\bibliography{EJOR-R3}

\end{document}